\documentclass[12pt]{amsart}
\usepackage{amsmath, calc}
\usepackage{multirow}
\usepackage{booktabs, array}
\usepackage{fancyhdr}
\usepackage{mdwlist}
\usepackage{xcolor}

\newtheorem{lemma}{Lemma}

\newtheorem*{euler}{Euler's Identity}
\newtheorem*{symmetry}{Symmetry}

\newtheorem*{main}{Main Theorem}
\newtheorem{proposition}{Proposition}

\theoremstyle{definition}
\newtheorem*{definition}{Definition}
\newtheorem*{properties}{Properties of anticontinuants}
\newtheorem{example}{Example}

\theoremstyle{remark}
\newtheorem*{remark}{Remark}
\newtheorem*{convention}{Convention}

\DeclareMathOperator{\len}{\mathrm{length}}
\newcommand{\OL}[1]{{\overleftarrow{#1} }}
\newcommand{\OR}[1]{{\overrightarrow{#1}}}

\newcommand{\ORfrak}{\OR{\mathfrak{q}}}
\newcommand{\OLfrak}{\OL{\mathfrak{q}}}

\begin{document}

\title{End-symmetric continued fractions and quadratic congruences}

\author{Barry Smith}
\email{barsmith@lvc.edu}
\address{Department of Mathematical Sciences \\ Lebanon Valley College\\
Annville, PA, USA}

\begin{abstract}
We show that for a fixed integer $n \neq \pm2$, the congruence $x^2 + nx \pm 1 \equiv 0 \pmod{\alpha}$ has the solution $\beta$ with $0 < \beta < \alpha$ if and only if $\alpha/\beta$ has a continued fraction expansion with sequence of quotients having one of a finite number of possible \emph{asymmetry types}.  This generalizes the old theorem that a rational number $\alpha/\beta > 1$ in lowest terms has a symmetric continued fraction precisely when $\beta^2 \equiv \pm 1 \pmod{\alpha}$.  
\end{abstract}

\maketitle 

\section{Introduction}\label{S:intro}


From relatively prime positive integers $\alpha > \beta$, we may form a simple continued fraction
\begin{equation*}
	\frac{\alpha}{\beta} = q_0 + \cfrac{1}{q_1 + \cfrac{1}{q_2 + \cfrac{1}{\ddots \, + \cfrac{1}{q_{s-1}}}}}
\end{equation*}
The positive integers $q_0$, \ldots $q_{s-1}$ can be computed as the quotients when the Euclidean algorithm is performed with $\alpha$ and $\beta$.  

Symmetric expansions like
\begin{equation*}
	\frac{25}{7} = 3 + \cfrac{1}{1 + \cfrac{1}{1 + \cfrac{1}{3}}}
\end{equation*}
naturally draw attention and have been studied since the mid-nineteenth century \cite{cH1848} \cite{jS1848} \cite{hS1855}. These early studies 
noted the following are equivalent:
\begin{enumerate}
	\item  $\beta^2 \equiv (-1)^{s+1} \pmod{\alpha}$ \label{L:congruence}
	\item  The quotient sequence is symmetric\footnote{when chosen appropriately -- see the bottom of p.2}, i.e., $q_0 = q_{s-1}$, $q_1 = q_{s-2}$, \ldots, \label{L:symmetry}
\end{enumerate}
Above, we see that $25/7$ has symmetric continued fraction expansion of even length, and we can readily confirm that $7^2 \equiv -1 \pmod{25}$.

The nominal goal of the first studies of symmetric continued fractions \cite{cH1848} \cite{jS1848} \cite{hS1855} was providing new constructive proofs that every prime congruent to $1$ modulo $4$ can be written as a sum of two squares.  Later works  \cite{jB1972} \cite{gC1908} \cite{HMW1990} focused exclusively on computation, crafting from the Euclidean algorithm procedures for representing numbers by quadratic forms.  Symmetry is notably absent.

Symmetry in continued fractions has arisen recently \cite{hC1996} \cite{DFP1982} \cite{jS1979} \cite{aP2001} in a different constellation of ideas centered around ``folded'' continued fractions.  These developments study mostly infinite continued fractions whose convergents have a form of iterated symmetry and explain the fantastic continued fraction expansions of certain ``naturally occurring'' irrational numbers.  They seem not to overlap much with the present work, with one significant exception to be indicated in Section \ref{S:looseends}.

The present work generalizes the equivalence of \eqref{L:congruence} and \eqref{L:symmetry}.  We begin by giving a flavor of the types of symmetry provided by the Main Theorem.


\begin{example}\label{Ex:simple}
The following are equivalent:
\begin{enumerate}
	\item $\beta$ is a root of one of the congruences $x^2  \pm x + 1 \equiv 0 \pmod{\alpha}$
	\item $\alpha/\beta$ has a continued fraction expansion with quotient sequence of the form 
\begin{equation*}
	\begin{matrix}
		q_0, & \ldots & q_{s-1}, & \textcolor{magenta}{q_s \pm 1}, & \textcolor{magenta}{q_s}, & q_{s-1}, & \ldots & q_0
	\end{matrix}
\end{equation*}
\end{enumerate}
Also, the following are equivalent:
\begin{enumerate}
	\item $\beta$ is a root of one of the congruences $x^2  \pm x - 1 \equiv 0 \pmod{\alpha}$
	\item $\alpha/\beta$ has a continued fraction expansion with quotient sequence of the form
\begin{equation*}
	\begin{matrix}
		q_0, & \ldots & q_{s-1}, & \textcolor{magenta}{q_s \pm 1}, & \textcolor{magenta}{1},   &\textcolor{magenta}{q_s}, & q_{s-1}, & \ldots & q_0
	\end{matrix}
\end{equation*}	
\end{enumerate}
And the following are equivalent:
\begin{enumerate}
	\item $\beta$ is a root of one of the congruences $x^2  \pm 3 x + 1 \equiv 0 \pmod{\alpha}$
	\item $\alpha/\beta$ has a continued fraction expansion with quotient sequence of the form
\begin{equation*}
	\begin{matrix}
		q_0, & \ldots & q_{s-1}, & \textcolor{magenta}{q_s \pm 3},  &\textcolor{magenta}{q_s}, & q_{s-1}, & \ldots & q_0
	\end{matrix}
\end{equation*}	
\end{enumerate}
In all three cases, the sign used in (1) will be the same as that used in (2) when $s$ is even and the opposite when $s$ is odd.
\end{example}

These are the simplest cases of the Main Theorem.  We see in each that $\beta$ is a root of a certain quadratic congruence modulo $\alpha$ precisely when it has a continued fraction expansion with ``end-symmetric'' sequence of quotients, that is, symmetric outside of an asymmetric core with a particular form (in magenta).  

Generally, the asymmetric core will have one of a finite number of possible forms.  Also, we must exclude a finite set of exceptional $\alpha$.

\begin{example}\label{Ex:complex}
Suppose that $\alpha \neq 2, 3$.  Then the following are equivalent:
\begin{enumerate}
	\item $\beta$ is a root one of of the congruences $x^2 \pm 4 x + 1 \equiv 0 \pmod{\alpha}$
	\item $\alpha/\beta$ has a continued fraction expansion with quotient sequence having one of the forms 
\begin{equation*}
	\begin{matrix}
		& q_0, & \ldots & q_{s-1}, & \textcolor{magenta}{q_s \pm 4},  &\textcolor{magenta}{q_s}, & q_{s-1}, & \ldots & q_0 \\
		q_0, & \ldots & q_{s-1}, & \textcolor{magenta}{q_s \pm 2}, & \textcolor{magenta}{1}, & \textcolor{magenta}{1},   &\textcolor{magenta}{q_s}, & q_{s-1}, & \ldots & q_0\\
		q_0, & \ldots & q_{s-1}, & \textcolor{magenta}{q_s + 1}, & \textcolor{magenta}{1}, & \textcolor{magenta}{2},   &\textcolor{magenta}{q_s}, & q_{s-1}, & \ldots & q_0\\
		q_0, & \ldots & q_{s-2}, & \textcolor{magenta}{q_s - 1}, & \textcolor{magenta}{2}, & \textcolor{magenta}{1},   &\textcolor{magenta}{q_s}, & q_{s-1}, & \ldots & q_0\\
	\end{matrix}
\end{equation*}
\end{enumerate}
\end{example}
We must exclude the cases $\alpha=2$ and $3$, since in the first case $\beta = 1$ and in the second case $\beta=1$ and $\beta = 2$ are roots of $x^2 \pm 4x + 1 \equiv 0 \pmod{\alpha}$, but $\alpha/\beta$ cannot be expanded as a continued fraction with one of the specified forms.

The above examples inspire the following terminology, which will facilitate the statement of the main theorem.

\begin{definition}
A finite asymmetric sequence $\OR{q}$ of positive integers can be uniquely written in the form 
\begin{equation*}
	\begin{matrix}
		q_0, & q_1, & \ldots & q_{s-1}, & q_s + (-1)^s c,  & \OR{x}, & q_s, & q_{s-1}, & \ldots & q_1, & q_0
	\end{matrix}
\end{equation*}
in which $c$ is a nonzero integer, $s$ is a nonnegative integer, and $\OR{x}$ is a sequence of positive integers.    We will say that $\OR{q}$ has \textbf{asymmetry type} $(c\phantom{c} ; \OR{x})$. A symmetric sequence of even length has asymmetry type $(0\phantom{c} ; \phantom{c})$, and a symmetric sequence of odd length with middle entry $x$ has asymmetry type $(0\phantom{c} ; \phantom{l}x)$.
\end{definition}

We must address how continued fraction expansions are chosen.  The numbers $\alpha/\beta$ we consider actually all have two continued fraction expansions, one with final quotient $1$ and the other with final quotient $\geq 1$.  For instance, in the example above, 25/7 can be expanded as the continued fraction with sequence of quotients $3$, $1$, $1$, $3$ or as that with sequence of quotients $3$, $1$, $1$, $2$, $1$.  We use the following convention throughout this work:

\begin{convention}
When a rational number is expanded as a simple continued fraction, the continued fraction will be chosen so that its final coefficient is 1 if and only if its initial quotient is 1.
\end{convention}

\begin{main}\label{T:main}
Fix integers $n$ and $s$ with $s=0$ or $1$.  When $s=0$, assume also that $n \neq \pm 2$.  Then there is a finite set $S$ of asymmetry types such that for all positive integers $\alpha$ outside of a finite set, the congruence
\begin{equation*}
	x^2 + nx + (-1)^s \equiv 0 \pmod{\alpha},
\end{equation*}
has the solution $\beta$ with $0 < \beta < \alpha$ if and only if the sequence of quotients of the simple continued fraction expansion of $\alpha/\beta$ has asymmetry type in $S$. 
\end{main}

The proof exploits properties of a special type of expression built from continuants.  We call these expressions anticontinuants and develop their properties in Section \ref{S:continuants}.  Section \ref{S:proof} is devoted to the proof of the Main Theorem.  Section \ref{S:looseends} discusses the exceptional case $n=2$ and $s=0$ in detail and includes a table of the possible asymmetry types that arise from the Main Theorem for small values of $n$.

\section{Continuants and anticontinuants}\label{S:continuants}

From a finite sequence of positive integers $\OR{q} = (q_0, \ldots, q_{s-1})$, we may compute a doubly indexed collection of numbers called \emph{continuants}.  

\begin{definition}
For $0 \leq i \leq j+2 \leq s+1$, we define the continuants $\ORfrak_{\! i,j}$ recursively:
\begin{equation*}
	\ORfrak_{\! i,i-2} = 0, \qquad \ORfrak_{\! i,i-1} = 1, \qquad \ORfrak_{\! i,j} = q_j \ORfrak_{\! i,j-1} + \ORfrak_{\! i,j-2} \quad \text{ for $j=i, \ldots, s-1$}    
\end{equation*}
When a more explicit description of the $q_i$'s is required, we will use the alternative notation:
\begin{equation*}
	\left[ q_i, \ldots, q_j \right] := \ORfrak_{\! i,j}\\[0.1cm]
\end{equation*}
\end{definition}

The connection with continued fractions is that if $\alpha/\beta$ has continued fraction expansion with sequence of quotients $q_0, \ldots, q_{s-1}$ and $\gcd(\alpha,\beta) = 1$, then 
\begin{equation}\label{E:numdenom}
	\alpha = \ORfrak_{\! 0,s-1} \qquad \text{and} \qquad \beta = \ORfrak_{\! 1,s-1}
\end{equation}
This is because the $\ORfrak_{\! i,s-1}$ satisfy the same recursion as the remainders in the Euclidean algorithm with $\alpha$ and $\beta$, starting with the final step and working backward.

Useful properties of continuants include
\begin{symmetry}
\begin{equation*}
	\left[ q_i, \ldots, q_j \right] = \left[ q_j, \ldots, q_i \right]
\end{equation*}
\end{symmetry}
\noindent and the remarkable
\begin{euler}
For $0 \leq k \leq  l \leq m+2$ and $m \leq n \leq s-1$,
\begin{align*}
	\ORfrak_{\! k,n} \ORfrak_{\! l,m} - \ORfrak_{\! k,m} \ORfrak_{\! l,n} = (-1)^{l+m+1} \ORfrak_{\! k,l-2} \, \ORfrak_{\! m+2,n}
\end{align*}
\end{euler}

Many proofs of these properties are known.  A streamlined method that proves both symmetry and Euler's identity simultaneously involves viewing the continuant $\ORfrak_{\! i,j}$ as the number of tilings by certain stackable tiles of a 1-dimensional board \cite{BQ2000}.


We could have defined continuants differently.  For instance, each is a polynomial in the numbers $q_0, \ldots, q_{s-1}$, and Euler gave an explicit description of the terms that appear.  Let us examine a few cases to get the general idea:
\begin{align*}
	\left[ q_0, q_1 \right] &= q_0 q_1 + 1\\
	\left[ q_0, q_1, q_2 \right] &= q_0 q_1 q_2 + q_0 + q_2\\
	\left[ q_0, q_1, q_2, q_3 \right] &= q_0 q_1 q_2 q_3 + q_0 q_1 + q_2 q_3 + q_0 q_3 + 1
\end{align*}
Euler observed that each term appearing in $\left[q_0, \ldots, q_{s-1} \right]$ may be produced by starting with the product $q_0 q_1 \cdots q_{s-1}$ and deleting some pairs of factors whose subscripts are consecutive integers. It can be shown that the number of such terms is $F_{s+1}$, the $s+1$st Fibonacci number, so $F_{s+1}$ is the minimal value for a continuant of length $s$ with positive integer entries.\\

We will consider expressions $\left[ q_{i}, \ldots, q_{j-1} \right] - \left[ q_{i+1}, \ldots, q_{j} \right]$ often enough that it is convenient to give them their own name and notation\footnote{The author regrets that the word ``alternant'' is already taken}.

\begin{definition}  
Given a sequence of integers $\OR{q} = (q_0, \ldots, q_{s-1})$, the \textbf{anticontinuant} $\ORfrak_{\! i,j}^{\star}$ is defined for $0 \leq i \leq j+1 \leq s$ by
\begin{equation*}
	\ORfrak_{\! i,j}^{\star} :=  \ORfrak_{\! i,j-1} - \ORfrak_{\! i+1,j} 
\end{equation*}
We also use the notations 
\begin{equation*}
	\left[ q_i, \ldots, q_j \right]^{\star} := \ORfrak_{\! i,j}^{\star} \quad \text{ and } \quad \left[ \OR{q} \right]^{\star} := \ORfrak_{\! 0,s-1}^{\star}
\end{equation*}
The anticontinuant of the sequence in reverse will be denoted
\begin{equation*}
	\OLfrak_{\! j,i}^{\star} := \left[ q_j, \ldots, q_i \right]^{\star} 
\end{equation*}
\end{definition}

When we combine the symmetry of continuants with Equation \eqref{E:numdenom}, we have the perhaps surprising observation that the continued fractions with sequences of quotients $q_0, \ldots, q_{s-1}$ and $q_{s-1}, \ldots, q_0$ have the same numerator (in lowest terms).  Then $\left| \ORfrak_{\! 0,s-1}^{\star} \right|$ is the distance between their denominators.

Anticontinuants have a recursive description that can be derived from that of continuants:
\begin{equation}\label{E:anticontrecursion}
	\ORfrak_{\! i,i-1}^{\star} = 0, \, \, \, \ORfrak_{\! i,i}^{\star} = 0, \quad \text{ and } \quad  \ORfrak_{\! i,j}^{\star} =  (q_i - q_j) \cdot \ORfrak_{\! i+1,j-1}- \ORfrak_{\! i+1,j-1}^{\star} \text{ for $i< j \leq s-1$}
\end{equation}

Explicit expressions for the first few nontrivial anticontinuants are
\begin{align*}
	\left[ q_0, q_1 \right]^{\star} &= q_0 - q_1\\
	\left[ q_0, q_1, q_2 \right]^{\star} &= q_0 q_1 - q_1 q_2\\
	\left[ q_0, q_1, q_2, q_3 \right]^{\star}  &= q_0 - q_1 + q_2 - q_3 + q_0 q_1 q_2 - q_1 q_2 q_3\\
	\left[ q_0, q_1, q_2, q_3, q_4 \right]^{\star}  &= q_0 q_1 - q_1 q_2 + q_2 q_3 - q_3 q_4 + q_0 q_3 - q_1 q_4 + q_0 q_1 q_2 q_3 - q_1 q_2 q_3 q_4
\end{align*}


We note the immediate properties:
\begin{properties} \hfill

\renewcommand{\theenumi}{\arabic{enumi}}
\begin{enumerate}
	\item  $[ q_0, \ldots, q_{s-1}, \OR{x}, q_{s-1}, \ldots, q_0]^{\star} = (-1)^{s} \left[ \OR{x} \right]^{\star}$.
	\item $\OLfrak_{\! j,i}^{\star} = - \ORfrak_{\! i,j}^{\star}$
\end{enumerate}
\end{properties}

Property (1) follows inductively from Equation \eqref{E:anticontrecursion}.  Property (2) can be proved using the symmetry of continuants in the definition of anticontinuant.  

We next involve the notion of asymmetry types defined in Section \ref{S:intro}.   We let $\len(\OR{q})$ be the number of entries in the sequence $\OR{q}$.
\begin{proposition}
Suppose the sequence of positive integers $\OR{q}$ has asymmetry type $(c\phantom{c} ; \OR{x})$.  Then
\begin{enumerate}
\setcounter{enumi}{2}
	\item $\left[ \OR{q} \right]^{\star} = 0$ if and only if $c=0$, i.e., $\OR{q}$ is symmetric
	\item If $\OR{q}$ is not symmetric, then the sign of $\left[ \OR{q} \right]^{\star}$ is the same as the sign of $c$
	\item If $\len(\OR{x}) = \lambda$, then \label{L:anticontinuantsize}
\begin{equation*}
	| c | \, F_{\lambda+1} \, \leq \, \left| \left[ \OR{q} \right]^{\star} \right| 
\end{equation*}
where $F_{\lambda+1}$ denotes the $\lambda+1$st Fibonacci number. 
\end{enumerate}
\end{proposition}

\begin{proof}
Property (1) shows that $\left[ \OR{q} \right]^{\star}=0$ if $\OR{q}$ is symmetric.  Property (3) will follow if we show that asymmetric anticontinuants are nonzero.  Properties (1) and (2) above reduce the proofs of this and Properties (4) and (5) to the case $s=0$.  We have $\left[ q+c, q \right]^{\star} = c$ and $\left[ q+c, x_0, q \right]^{\star} = c x_0$ and the properties are clear in these cases.  So consider an anticontinuant $\left[ q+c, x_0, \ldots, x_{\lambda-1}, q \right]^{\star}$ with $\lambda \geq 2$.

Using the recursion for anticontinuants, Equation \eqref{E:anticontrecursion}, we have
\begin{equation*}
	\left[ q+c, x_0, \ldots, x_{\lambda-1}, q \right]^{\star} = c \left[ x_0, \ldots, x_{\lambda-1} \right] - \left[ x_0, \ldots, x_{\lambda - 2} \right] + \left[ x_1, \ldots, x_{\lambda-1} \right]
\end{equation*}
If $c$ is positive, then the first term is at least the second, so the right side is positive.  If $c$ is negative, then the sum of the first and last terms is negative, so the right side is negative.  This proves Properties (3) and (4).

To prove Property (5) in the case $s=0$, it is enough to consider the case $c > 0$.  Using the recursive for continuants, we may then rewrite the equation in the paragraph above as
\begin{equation}\label{E:boundingeq}
	\left[ q+c, x_0, \ldots, x_{\lambda-1}, q \right]^{\star} = \left( c x_{\lambda-1} - 1 \right) \left[x_0, \ldots, x_{\lambda-2} \right] + c \left[ x_0, \ldots, x_{\lambda-3} \right] + \left[ x_1, \ldots, x_{\lambda-1} \right] 
\end{equation}
Because the minimal value taken by a continuant of length $s$ is $F_{s+1}$, we have
\begin{align*}
	\left[ q+c, x_0, \ldots, x_{\lambda-1}, q \right]^{\star} &\geq (c-1) F_{\lambda} + c F_{\lambda-1} + F_{\lambda} =  c F_{\lambda+1} \qedhere
\end{align*}
\end{proof}

  Because continuants with positive integer entries increase when either an entry is increased or the number of entries increases, the continuants that evaluate to a fixed integer are finite in number. Property (1) shows this is not true for anticontinuants.  We can salvage this property by ignoring symmetric ends of sequences.

\begin{proposition}\label{P:finitetypes}
Let $n \neq \pm2$ be an integer.  The sequences with anticontinuant equal to $n$ have only a finite number of asymmetry types.  The sequences with anticontinuant equal to  $2$ have asymmetry type $(2\phantom{c} ; \phantom{c})$, $(1\phantom{c} ; \phantom{c}2)$, $(2\phantom{c} ; \phantom{c}1)$,  or $(1\phantom{c} ;  \phantom{c}p,1)$ where $p$ is an arbitrary positive integer, and those with anticontinuant equal to $-2$ are their reversals.
\end{proposition}

\begin{proof}
Using properties (1), (2), (3), and (4) of anticontinuants, it is enough to consider a positive integer $n$ and anticontinuants of the form $\left[ q+c, x_0, \ldots, x_{\lambda-1}, q \right]^{\star}$ with $c  > 0$.  We must show there are only finitely many such anticontinuants equal to $n$. Property (5) of anticontinuants shows at least that the $\lambda$'s of such anticontinuants are bounded.

Note that  $\left[ q+c, q \right]^{\star} = c$ and $\left[ q+c, x_0, q \right]^{\star}=c x_0$. Thus, anticontinuants with $\lambda = 0$ or $1$ that evaluate to $n$ have only finitely many asymmetry types.

Now fix $\lambda \geq 2$.  If $c$ and $x_{\lambda-1}$ are both larger than $n$, then each of the three terms on the right side of Equation \eqref{E:boundingeq} is greater than $n$ for every choice of positive integers $x_0, \ldots, x_{\lambda-2}$.  Thus, there are finitely many pairs of positive integers $c$, $x_{\lambda-1}$ for which $\left[ q+c, x_0, \ldots, x_{\lambda-1}, q \right]^{\star} = n$ is possible.  

If $c$ and $x_{\lambda-1}$ is such a pair, then when the right side of Equation \eqref{E:boundingeq} is expanded and considered as a polynomial in $x_0$, \ldots, $x_{\lambda-2}$, all terms appear with positive coefficient.  In addition, when one of the following conditions holds
\begin{itemize}
	\item $\lambda \geq 3$
	\item $\lambda = 2$ and $c \geq 2$
	\item $\lambda = 2$ and $x_{\lambda-1} \geq 2$
\end{itemize}
each of $x_0$, \ldots, $x_{\lambda-2}$ appears as a factor in one of these terms.  If, instead, $\lambda = 2$ and $c=x_{1}=1$, then the anticontinuant has the form $\left[ q+1, x_0, 1, q \right]^{\star}$ which one can directly check is equal to $2$.  

Thus, if $n \neq 2$, then for each of the finitely many $c$ and $x_{\lambda-1}$ for which $\left[ q+c, x_0, \ldots, x_{\lambda-1}, q \right]^{\star}$ can possibly equal $n$, the positive integers $x_0$, \ldots, $x_{\lambda-2}$ can be chosen in only finitely many ways to accomplish this.  Because there are also only finitely many possible lengths $\lambda$, the first statement of the proposition is proved. Property (5) shows that if $\left[ q+c, x_0, \ldots, x_{\lambda-1}, q \right]^{\star} = 2$, then $\lambda \leq 1$.  One can check that the asymmetry types given in the proposition are the only ones producing an anticontinuant equal to $\pm 2$.
\end{proof}

\begin{proposition}\label{P:CFcongruence}
Suppose $\alpha$ and $\beta$ are relatively prime positive integers whose simple continued fraction expansion has sequence of quotients $\OR{q}$.  Then $\beta$ is a root of the quadratic congruence
\begin{equation*}
	x^2 + \left[ \OR{q} \right]^{\star} \, x + (-1)^{\len (\OR{q})} \equiv 0 \pmod{\alpha}
\end{equation*}
\end{proposition}

\begin{proof}
Write $\OR{q} = q_0, \ldots, q_{s-1}$ and recall that $\beta = \ORfrak_{\! 1,s-1}$ and $\left[ \OR{q} \right]^{\star} = \ORfrak_{\! 0,s-2} - \ORfrak_{\! 1,s-1} = \ORfrak_{\! 0,s-2} - \beta$.  Then
\begin{equation*}
	\beta^2 + \left[ \OR{q} \right]^{\star} \, \beta = \ORfrak_{\! 0,s-2}  \ORfrak_{\! 1,s-1} 
\end{equation*}
The identity 
\begin{equation*}
	 \ORfrak_{\! 0,s-1} \ORfrak_{\! 1,s-2} - \ORfrak_{\! 0,s-2}  \ORfrak_{\! 1,s-1} = (-1)^{s}
\end{equation*}
expresses a well-known relationship between successive convergents of a continued fraction and can be obtained as the specialization of Euler's continuant identity with $k=0$, $l=1$, $m=s-2$, and $n=s-1$.  Noting that $\alpha = \ORfrak_{\! 0,s-1}$, the proposition follows. 
\end{proof}

\begin{remark}
The proof is valid regardless of which of the two continued fraction expansions we choose for $\alpha/\beta$.  From them we get the two unique congruences $x^2 + mx + 1 \equiv 0 \pmod{\alpha}$ and $x^2 + nx - 1 \equiv \pmod{\alpha}$ satisfied by $\beta$. 
\end{remark}

The following lemma puts a bound on $\left| [\ORfrak ]^{\star} \right|$.  The convention about which continued fraction expansion to choose now becomes important -- the bound always holds under our convention.  That bound, in turn, is a crucial ingredient in the proof of the Main Theorem.

\begin{proposition}\label{P:anticontinuantsize}
If $\OR{q} = q_0, \ldots, q_{s-1}$ is a sequence of positive integers with either $q_0, q_{s-1} \geq 2$ or $q_0 = q_{s-1} = 1$, then $\left| \left[ \OR{q} \right]^{\star} \right| < \tfrac{1}{2} \ORfrak_{\! 0,s-1}$.
\end{proposition}

\begin{proof}
If $s=1$ the lemma is immediate, so suppose $s \geq 2$.  Euler's continuant identity with $k=m=0$, $l=1$, and $n=s-1$ shows that $\ORfrak_{\! 0,s-1} = q_0 \ORfrak_{\! 1,s-1} + \ORfrak_{\! 2,s-1}$.  This and the recursive definition of continuant give
\begin{align*}
	0 < \ORfrak_{\! 1,s-1} &= \frac{1}{q_0} \left( \ORfrak_{\! 0,s-1}- \ORfrak_{\! 2,s-1} \right) <  \frac{1}{2} \ORfrak_{\! 0,s-1} \\
	0 < \ORfrak_{\! 0,s-2} &= \frac{1}{q_{s-1}} \left(\ORfrak_{\! 0,s-1}- \ORfrak_{\! 0,s-3} \right) < \frac{1}{2} \ORfrak_{\! 0,s-1}
\end{align*}	
The lemma follows if we recall that $\left[ \OR{q} \right]^{\star} = \ORfrak_{\! 0,s-2}- \ORfrak_{\! 1,s-1}$.

Otherwise, if $q_0 = q_{s-1} = 1$, then the same identities show that
\begin{align*}
	\frac{1}{2} \ORfrak_{\! 0,s-1} &= \frac{1}{2} \left( \left[ q_1, \ldots, q_{s-2}, 1 \right] + \left[ q_2, \ldots, q_{s-2}, 1 \right] \right) < \left[ q_1, \ldots, q_{s-2}, 1 \right] < \ORfrak_{\! 0,s-1}\\
	\frac{1}{2} \ORfrak_{\! 0,s-1} &= \frac{1}{2} \left( \left[ 1, q_1, \ldots, q_{s-2} \right] + \left[ 1, q_1, \ldots, q_{s-3} \right] \right) < \left[ 1, q_1, \ldots, q_{s-2} \right] < \ORfrak_{\! 0,s-1}
\end{align*}
The lemma follows since  $\left[ \OR{q} \right]^{\star} = \left[ 1, q_1, \ldots, q_{s-2} \right] - \left[ q_1, \ldots, q_{s-2}, 1\right]$.  
\end{proof}

\section{Proof of the Main Theorem}\label{S:proof}

The proof of the Main Theorem almost falls out of the properties in the previous section.  If $\beta$ is a root of $x^2 + nx + (-1)^s \equiv 0 \pmod{p}$, then Proposition \ref{P:CFcongruence} shows it is also a root of $x^2 + \left[ \OR{q} \right]^{\star} + (-1)^s \equiv 0 \pmod{\alpha}$, in which $\OR{q}$ is the sequence of quotients of the continued fraction for $\alpha/\beta$.  But this forces $n$ and $\left[ \OR{q} \right]^{\star}$ to be congruent modulo $\alpha$.  If $\alpha$ is large enough, then the bound from Proposition \ref{P:anticontinuantsize} forces them to be equal and the theorem follows from Proposition \ref{P:finitetypes}.  The only problem with this argument is that the application of Proposition \ref{P:CFcongruence} assumes that the length of $\OR{q}$ is $s$.  While this usually can be arranged, the application of Proposition \ref{P:anticontinuantsize} requires that $\OR{q}$ is chosen with our convention, so the flexibility is gone.  The next two lemmas show that the flexibility is, outside of a finite number of cases, not needed. 

The first lemma has inherent interest.  It explains by itself the parities of the lengths of the quotient sequences appearing in Examples 1 and 2 in Section \ref{S:intro}.
\begin{lemma}\label{L:source}
Let $u$ and $v$ be relatively prime integers, $u > v > 0$.  Let $v^{-1}$ be the smallest positive inverse of $v$ modulo $u$.  The length of the simple continued fraction expansion of $u/v$ is odd if $v$ and $v^{-1}$ are on the same side of $u/2$ and even if they are on opposite sides. (Recall the convention that the continued fraction expansion is chosen so that the final quotient is 1 if and only if the initial quotient is 1.)
\end{lemma}

\begin{proof}
By ``same side'', we mean $v$ and $v^{-1}$ are either both $\leq u/2$ or both $> u/2$. We first reduce to the case where $v < u/2$.  If $v=u/2$, then $u=2$ and $v=1$ and the lemma is clear.   If $v > u/2$, then the continued fraction expansion of $u/v$ starts with $1$, and our convention has us choose the one that ends in $1$ as well.  Suppose the sequence of quotients is $(1, q_1, \ldots, q_{s-2}, 1)$. The sequence of quotients of $u/(u-v)$ is then $(q_1+1, q_2, \ldots, q_{s-2}+1)$.  The lengths of both sequences thus have the same parity. Observing that the inverse of $u-v$ is $u-v^{-1}$, we see that $v$ and $v^{-1}$ are on the same side of $u/2$ if and only if $u-v$ and $(u-v)^{-1}$ are. The reduction is complete.

To show the lemma holds when $v < u/2$, we use induction on the length of the continued fraction expansion of $u/v$. If the length is $1$, then $v=1$ and the lemma is immediate. Assume the proposition is true whenever the length is s. Fix a pair $u$, $v$ for which the sequence of quotients is $(q_0, \ldots, q_s)$.  Let $k$ be the integer for which $v v^{-1} = 1 + ku$.  Since $v^{-1} < u$, we have $k < v$.  It follows that $k = v - u^{-1}$, where $u^{-1}$ is the smallest positive inverse of $u$ modulo $v$.     

By assumption, $q_0 \neq 1$, so by our convention $q_s \neq 1$ as well.  Assume for now that $q_1 \neq 1$.  Then $u = q_0 v + r$ where $v/r$ has continued fraction expansion with sequence of quotients $(q_1, \ldots, q_s)$ and $r  \leq v/2$. The equation $u = q_0 v + r$ shows that $u^{-1}$ is also the smallest positive inverse of $r$ modulo $v$. If $s$ is odd, then by the induction hypothesis $u^{-1} \leq \frac{v}{2}$, so $k \geq \tfrac{v}{2}$.  It follows that $v^{-1} > \tfrac{u}{2}$ while the continued fraction expansion of $u/v$ has even length.  Similarly, if $s$ is even, then $k < \tfrac{v}{2}$. It follows that $v^{-1} < \tfrac{1}{v} + \tfrac{u}{2}$.  But the continued fraction expansion of $u/v$ takes $s+1 \geq 2$ steps, so $v \geq 2$.  We conclude that $v^{-1} \leq \tfrac{u}{2}$ while the continued fraction expansion of $u/v$ has odd length.

If $q_1 = 1$, then the previous paragraph must be modified.  Since $q_0$ and $q_s$ are greater than $1$, we must have $s \geq 2$.  The continued fraction expansion of $v/r$ must by our convention be chosen to be $(1,q_2,\ldots, q_s-1,1)$.  Also, $r > v/2$.  If $s$ is odd, then by the  induction hypothesis $u^{-1} \leq v/2$, and if $s$ is even, then $u^{-1} > v/2$.  The rest of the argument follows as before.
\end{proof}

\begin{lemma}\label{L:finiteset}
Fix integers $n$ and $s$ with $s = 0$ or $1$. When $s=0$, assume also that $n \neq \pm 2$.  For all $\alpha$ outside a finite set of positive integers, every solution $\beta$ of the congruence $x^2 + nx + (-1)^s \equiv 0 \pmod{\alpha}$ with $0 < \beta < \alpha$ is such that the parity of the length of the continued fraction expansion of $\alpha/\beta$ is equal to that of $s$.
\end{lemma}

\begin{proof}
Suppose first that $s=0$.  Appealing to Lemma \ref{L:source}, we must show that there are only finitely many $\alpha$ for which  $x^2 + nx + 1 \equiv 0 \pmod{\alpha}$ has a root $\beta$ with smallest positive inverse $\beta^{-1}$ on the same side of $\alpha/2$ as $\beta$.  We will see that this is the case if we exclude the finitely many $\alpha$ for which one of the following holds:
\begin{enumerate}
	\item $\alpha \leq 2|n|$
	\item One of the congruences $\gamma(\gamma-|n|) \equiv -1 \pmod{\alpha}$ with $\gamma=1, 2, \ldots, |n|-1$ is valid
	\item One of the congruences $\eta(\eta-2|n|) \equiv -4 \pmod{\alpha}$ with $\eta = 1, \ldots, 2|n|-1$ is valid
\end{enumerate}
We are using the assumption that $n \neq \pm 2$ in saying that there are finitely many $\alpha$ for which (2) does not hold.  Note also that $\eta(\eta-2n) = -4$ only when $\eta = n \pm \sqrt{n^2-4}$.  Again, our assumption that $n \neq \pm 2$ ensures this is not an integer, so only finitely many $\alpha$ satisfy (3).

Assume for now that $\beta \leq \alpha/2$.  Since $\beta (\beta+n) \equiv -1 \pmod{\alpha}$, we see that $\beta^{-1} \equiv \alpha - \beta - n \pmod{\alpha}$. Because we have excluded $\alpha$ for which (2) holds, when $n < 0$ we have $\beta > |n|$.  Thus,  $\alpha - \beta - n < \alpha$.  Because we have also excluded $\alpha$ satisfying (1), it follows that
\begin{equation*}
	\beta^{-1} = \alpha - \beta - n
\end{equation*}
  We must show that $\beta+ n \geq \alpha/2$, so $\beta$ and $\beta^{-1}$ are on the same side of $\alpha/2$, for only finitely many $\alpha$. 

Suppose that $\beta+ n \geq \alpha/2$. In this case, $n > 0$.  Since $\beta \leq \alpha/2$, it follows that $-2 \beta$ is congruent to one of the integers in the interval $(0, 2n)$.  Suppose this integer is $\eta$.  We also have $\eta(\eta-2n) \equiv -4 \pmod{\alpha}$.   Thus, $\beta+n < \alpha/2$ since we exclude the $\alpha$ satisfying (3).

Now suppose that $\beta > \alpha/2$.  Then $\alpha-\beta$ is a root of $x^2-nx+1\equiv0 \pmod{\alpha}$ with $0 < \alpha-\beta < \alpha/2$.  Also, $(\alpha-\beta)^{-1} = \alpha - \beta^{-1}$.  The above work shows that, after excluding $\alpha$ satisfying one of conditions (1)-(3), it must be that  $\alpha- \beta$ and $\alpha-\beta^{-1}$ are on opposite sides of $\alpha/2$.  Thus $\beta$ and $\beta^{-1}$ are as well, and there are no additional values of $\alpha$ to exclude in this case.  The lemma then holds when $s=0$. 

The proof when $s=1$ is very similar.  We leave the details to the reader, noting that in this case we must exclude $\alpha$ for which one of the following holds:
\begin{enumerate}
	\item $\alpha \leq 2|n|$
	\item One of the congruences $\gamma(\gamma-|n|) \equiv 1 \pmod{\alpha}$ with $\gamma=1, 2, \ldots, |n|-1$ is valid
	\item One of the congruences $\eta(\eta-2|n|) \equiv 4 \pmod{\alpha}$ with $\eta=1, \ldots, 2|n|-1$ is valid
\end{enumerate}

\end{proof}

\begin{remark} 
The above proof indicates that for a specific $n$, it is possible to determine the finitely many exceptional $\alpha$.  Let us consider, for instance, when $\beta$ is a root of a congruence $x^2 \pm 3x - 1 \equiv 0 \pmod{\alpha}$.  The proof shows the exceptional $\alpha$ must satisfy one of the following:
\begin{itemize}
	\item $\alpha \leq 6$
	\item $\gamma(\gamma-3) \equiv 1 \pmod{\alpha}$ with $\gamma=1$ or $2$
	\item $\eta(\eta - 6) \equiv 4 \pmod{\alpha}$ for some $\eta$ in $\{ \, 1, 2, 3, 4, 5 \, \}$
\end{itemize}
The set of such $\alpha$ is $\{ \, 1, 2, 3, 4, 5, 6, 9, 12, 13 \, \}$.  We can check each of these moduli individually to see if there are any for which $x^2 \pm 3x - 1 \equiv 0 \pmod{\alpha}$ has a solution $\beta$.  We find the pairs $(\alpha,\beta) = (3,1)$, $(3,2)$, $(9,2)$, $(9,4)$,  $(9,5)$, $(9,7)$, $(13,5)$, and $(13,8)$.
\end{remark}

We are now positioned to prove the Main Theorem from Section \ref{S:intro}.

\begin{proof}[Proof of the Main Theorem]
Choose $\alpha$ to be greater than $2|n|$ and one of the $\alpha$ for which the conclusion of Lemma \ref{L:finiteset} holds.  We are thus avoiding a finite number of $\alpha$. By assumption, 
\begin{equation*}
	\beta^2 + n \beta + (-1)^s \equiv 0 \pmod{\alpha}.
\end{equation*}  If $\OR{q}$ is the sequence of quotients of the continued fraction expansion of $\alpha/\beta$, then Proposition \ref{P:CFcongruence} shows also
\begin{equation*}
	\beta^2 + \left[ \OR{q} \right]^{\star} \, \beta + (-1)^s \equiv 0 \pmod{\alpha}
\end{equation*}
These two congruences force
\begin{equation*}
	 \left[ \OR{q} \right]^{\star} \equiv n\pmod{\alpha}
\end{equation*}
By assumption, $n$ is in the interval $\left( -\alpha/2, \alpha/2 \right)$.  Lemma \ref{L:anticontinuantsize} shows that $\left[ \OR{q} \right]^{\star}$ is in the same interval.  Thus, $ \left[ \OR{q} \right]^{\star}=n$.  The theorem now follows from Proposition \ref{P:finitetypes}.
\end{proof}

\section{Loose ends}\label{S:looseends}

\subsection{Explicit Examples}

For specific $n$, we can sharpen the statement of the Main Theorem as in Examples \ref{Ex:simple} and \ref{Ex:complex} in Section \ref{S:intro}. For example, suppose that $\alpha > \beta > 0$ are integers and $\beta$ is a root of one of the congruences $x^2 \pm 4x + 1 \equiv 0 \pmod{\alpha}$.  Let $q_0, \ldots, q_{s-1}$ be the sequence of quotients of the continued fraction expansion of $\alpha/\beta$ chosen with our usual convention.  The proof of the Main Theorem shows that if $\alpha$ is outside of a finite exceptional set of moduli, we have $\left[ \OR{q} \right]^{\star}  = \pm 4$.  Moreover, it shows that the exceptional set consists of those $\alpha$ such that either $\alpha \leq 8$ or the conclusion Lemma \ref{L:finiteset}.  Working as in the remark following the lemma, we find that such $\alpha$ must be in the set $\{ \, 1, 2, 3, 4, 5, 6, 7, 8, 11, 12 \, \}$. Solving the congruences $x^2 \pm 4x + 1 \equiv 0 \pmod{\alpha}$, we find the pairs $(\alpha,\beta) = (2,1)$, $(3,1)$, $(3,2)$, $(6,1)$, $(6,5)$, $(11,3)$, $(11,4)$, $(11,7)$, and $(11,8)$.  

Suppose that $(c\phantom{c} ; \phantom{l}\OR{x})$ is the asymmetry type of a sequence with even length and anticontinuant equal to $\pm 4$.  The particular sequence 
\begin{equation*}
	c+1, \OR{x}, 1
\end{equation*}
 has this asymmetry type -- say it is the sequence of quotients of the continued fraction expansion of $\alpha/\beta$. Then $\alpha$  will be in the list of exceptional $\alpha$ in the proof of the Main Theorem.  This is because, as written, the sequence $c+1, \OR{x}, 1$ violates our convention for choosing continued fractions.  Choosing the continued fraction with odd length instead, the pair $(\alpha,\beta)$ now fails the conclusion of Lemma \ref{L:finiteset}.  Thus, \emph{we may look through the set of pairs $(\alpha,\beta)$ in the previous paragraph to find all of the possible asymmetry types of even length with anticontinuant equal to 4}.  This is an alternative to trying to determine them directly from the definition and recursion of anticontinuants, and it is the method used to construct the table below.

Computing the continued fractions of even length for $\alpha/\beta$ in the list above, we find the following correspondences between fractions $\alpha/\beta$ and sequences of quotients:
\begin{align*}
	\frac{6}{1} &\longleftrightarrow 5, 1 \qquad \frac{6}{5}  \longleftrightarrow  1,5\\
	\frac{11}{3} & \longleftrightarrow  3,1,1,1 \qquad \frac{11}{7}  \longleftrightarrow 1,1,1,3\\
	\frac{11}{4} & \longleftrightarrow  2,1,2,1  \qquad \frac{11}{8}  \longleftrightarrow  1,2,1,2
\end{align*}
We can easily check that these have anticontinuant equal to $\pm 4$. The fractions $\alpha/\beta = 2/1,$ $3/1$, and $3/2$ do not have sequences of quotients with anticontinuant equal to $\pm 4$ and so are true exceptions.  We thus find the entire list of asymmetry types with anticontinuant equal to $4$ is $(4\phantom{c} ; \phantom{c})$, $(2\phantom{c} ;  \phantom{l}1,1)$, and $(1\phantom{c} ; \phantom{l}1,2)$.

If a sequence has asymmetry type $(c\phantom{c} ; \phantom{l}\OR{x})$, let us call $c$ its \textbf{marginal asymmetry} and $\OR{x}$ its \textbf{core asymmetry}.    We may obtain the asymmetry types of sequences with anticontinuant equal to $-4$ by negating the marginal asymmetries and reversing the core asymmetry sequences of the asymmetry types with anticontinuant equal to $4$.  We have thus verified the statement made in Example \ref{Ex:complex} in Section \eqref{S:intro}.

The following table provides the asymmetry types with anticontinuant equal to a positive integer $n \leq 6$, classified by the parity of the length of the core asymmetry.  Exceptional $(\alpha,\beta)$ are those for which $\beta^2 \pm n\beta + (-1)^s \equiv 0 \pmod{\alpha}$ but $\alpha/\beta$ does not have continued fraction whose quotient sequence has anticontinuant equal to $n$.  

\newpage

\renewcommand{\arraystretch}{1.25}

\footnotesize
\vspace{-2cm}


\begin{tabular}{c c c c}
(Anticont. Value,Length)  &  Marginal Asymmetry & Core Asymmetry & Exceptional $(\alpha,\beta)$\\ \hline
(1, even) & 1 &  & None\\ \hline
(1, odd) & 1 & 1 & None\\ \hline
(2, even) & 2 &  & None\\ 
 & 1 & p,1 & \\ \hline
 (2, odd) & 1 & 2 & (2,1) \\
  & 2 & 1 & \\ \hline
  (3, even) & 3 &  & None \\ \hline
  (3, odd) & 1 & 3 & (3,1)\\
   & 3 & 1 & (3,2)\\
   & 1 & 1,1,1 & \\ \hline
  (4, even) & 4 & & (2,1)\\
    & 2 & 1,1 & (3,1) \, (3,2) \\
    & 1 & 1,2 & \\ \hline
  (4, odd) & 1 & 4 & (2,1)\\
    & 2 & 2 & (4,1), (4,3)\\
    & 4 & 1 & (5,2), (5,3) \\
    & 1 & 1,2,1 &\\
    & 1 & 2,1,1 &  \\ \hline
  (5, even) & 5 & & (3,1), (3,2)\\
   & 1 & 2,2 & (5,2), (5,3) \\
   & 2 & 2,1 & (7,1), (7,6) \\
   & 1 & 1,1,1,1 & \\ \hline
  (5, odd) & 1 & 5 & (5,1), (5,4) \\ 
    & 5 & 1 & (7,2), (7,5)\\
    & 1 & 1,3,1 & (7,3), (7,4) \\ 
    & 1 & 3,1,1 & \\
    & 1 & 2,2,1 & \\ \hline
  (6, even) & 6 & & (2,1) \\
  & 3 & 1,1 & (4,1), (4,3) \\
  & 1 & 1,3 & (7,2), (7,5) \\
  & 2 & 3,1 & (7,3), (7,4) \\
  & 1 & 3,2 & (8,3), (8,5)\\
  & 1 & 2,1,1,1 &  \\
   & 1 & 1,1,2,1 & \\ \hline
   (6, odd) &  1 & 6 & (2,1)\\
   & 2 & 3 & (3,1), (3,2)\\
   & 3 & 2 & (5,2), (5,3)\\
   & 6 & 1 &  (6,1), (6,5)\\
   & 1 & 1,4,1 & (9,2), (9,7)\\
   & 1 & 4,1,1 & (9,4), (9,5)\\
   & 2 & 1,1,1 & (10,3), (10,7)\\
   & 1 & 1,1,2 & \\
   & 1 & 2,3,1 & \\
   & 1 & 3,2,1 & 
\end{tabular}\\[1cm]

\normalsize

\subsection{The case $n=\pm 2$ and $s=0$}

We now examine the pairs $(\alpha,\beta)$ for which $\beta$ is a root of $x^2 \pm 2x + 1 = (x\pm 1)^2 \equiv 0 \pmod{\alpha}$.  Say $(x \pm 1)^2 = \gamma \alpha$.  If we write $\alpha = b n^2$, where $b$ is the square-free part of $\alpha$, then $\gamma = b a^2$, say, and $x = ban \mp 1$.  We are thus studying fractions of the form $bn^2/(ban \mp 1)$.  Analyzing as in the beginning of this section, one finds that without exception they correspond to sequences of quotients of one of the forms
\begin{equation*}
	\begin{matrix}
		& q_0, & \ldots & q_{s-1}, & \textcolor{magenta}{q_s \pm 2},  &\textcolor{magenta}{q_s}, & q_{s-1}, & \ldots & q_0 \\
		q_0, & \ldots & q_{s-1}, & \textcolor{magenta}{q_s + 1}, & \textcolor{magenta}{x}, & \textcolor{magenta}{1},   &\textcolor{magenta}{q_s}, & q_{s-1}, & \ldots & q_0\\
		q_0, & \ldots & q_{s-1}, & \textcolor{magenta}{q_s - 1}, & \textcolor{magenta}{1}, & \textcolor{magenta}{x},   &\textcolor{magenta}{q_s}, & q_{s-1}, & \ldots & q_0\\
	\end{matrix}
\end{equation*}
in which $x$ is an arbitrary positive integer.

All three forms appear in \cite{jS1979}.  Specifically,  the author shows in the proof of Theorem 1 that continued fractions for numbers of the form $\tfrac{n^2}{an+1}$ (the case b=1 above)  have continued fractions with the first form above.  In Theorem 11 he also essentially states that the latter two forms appear as the continued fraction expansions of numbers of the form $\tfrac{ban+1}{bn^2}$, but only for a specific $b$.  In this and subsequent papers, concerning ``folded'' continued fractions, interest has been in iterating the form to create interesting infinite continued fractions.

Continued fractions for the general case $bn^2/(ban \pm 1)$ were studied in \cite{DLS}.  That work gives a criterion for determining which of the above forms is the one.  If $\alpha= bn^2$, $\beta = ban \mp 1$, 
and $d = \gcd(a,n)$, then we may factor $d^2$ from $n^2$ and from $an$ and incorporate it as a factor of $b$.  We may thus assume that $\alpha = bn^2$ and $\beta = ban \mp 1$ with relatively prime $n$ and $a$, and $a$, $b$, and $n$ are uniquely determined in this fashion.   

The sequence of quotients of $\alpha/\beta$ will have the first form above if $b=1$.  If $b \geq 2$, it has one of the other forms with $x = b-1$.  This was proved by giving explicit formulas for the remainders of the Euclidean algorithm with $\alpha$ and $\beta$ in terms of those of the Euclidean algorithm with $n$ and $a$.  Furthermore, it was shown that when $b=1$, the first remainder less than $n$ appearing is an inverse for $a$ modulo $n$, giving a new algorithm for computing inverses in modular arithmetic.  Certainly further exploration should be made of the remainders in the Euclidean algorithm with the pairs $(\alpha,\beta)$ under consideration in the present work.



\begin{thebibliography}{99}

\bibitem{BQ2000}
A. Benjamin, J. Quinn, and F. Su, \emph{Counting on continued fractions}, Math. Mag. 73 no. 2 (2000), 98--104.


\bibitem{jB1972}
J. Brillhart, \emph{Note on representing a prime as a sum of two squares},
Math. Comput. 26 (1972), 1011--1013.

\bibitem{hC1996}
H. Cohn, \emph{Symmetry and specializability in continued fractions}, Acta Arith. 75 no. 4 (1996), 297--320.

\bibitem{gC1908}
G. Cornacchia, \emph{Su di un metodo per la risoluzione in numeri interi dell' equazione $\sum_{h=0}^{n} C_h x^{n-h} = P$}, Giornale di Matematiche di Battaglini 46 (1908), 33--90.


\bibitem{DFP1982}
M. Dekking, M. Mend\'{e}s France, and A. van der Poorten, \emph{FOLDS! II. Symmetry disturbed}, Math. Intelligencer {4} (1982), 173--181.

\bibitem{DLS}
C. Doran, S. Lu, B. Smith, \emph{A new algorithm for computing inverses in modular arithmetic}, preprint, 10 pages.


\bibitem{HMW1990}
K. Hardy, J. B. Muskat, and K. S. Williams, \emph{A deterministic algorithm for solving $n=fu^2+gv^2$ in coprime integers $u$ and $v$}, {Math. Comp.} {55} (1990), 327--343.

\bibitem{cH1848}
C. Hermite, \emph{Note au sujet de l'article pr\'ec\'edent}, {J. Math. Pures Appl.} (1848), 15.


\bibitem{jS1848}
J.-A. Serret, \emph{Sur un th\'eor\`eme r\'elatif aux nombres enti\`eres}, J. Math. Pures Appl. (1848),  12--14.

\bibitem{jS1979}
J. Shallit, \emph{Simple continued fractions for some irrational numbers}, J. Number Theory {11} (1979), 209--217.

\bibitem{hS1855}
H. J. S. Smith, \emph{De compositione numerorum primorum $4\lambda + 1$ ex duobus quadratis}, {J. Reine Angew. Math.} {50} (1855), 91--92.

\bibitem{aP2001}
A. van der Poorten, \emph{Symmetry and folding of continued fractions}, {J. Th\'{e}or. Nombres Bordeaux} {13} (2001), 69--77.

\end{thebibliography}
\end{document}